\documentclass{amsart}
\usepackage{amssymb, amsmath, latexsym, mathabx, dsfont,tikz}
\usepackage{multirow}
\usepackage{setspace}
\usepackage{enumerate}
\usepackage{diagbox}
\usepackage{stmaryrd}

\renewcommand{\baselinestretch}{\baselinestretch}
\renewcommand{\baselinestretch}{1.1}
\numberwithin{equation}{section}

\newtheorem{thm}{Theorem}[section]
\newtheorem{lem}[thm]{Lemma}

\newtheorem{prop}[thm]{Proposition}

\theoremstyle{definition}

\theoremstyle{remark}
\newtheorem{rmk}[thm]{Remark}

\numberwithin{equation}{section}

\newcommand{\ra}{{\ \longrightarrow \ }}
\newcommand{\nra}{{\ \longarrownot\longrightarrow \ }}
\newcommand{\gen}{\text{gen}}

\newcommand{\ord}{\text{ord}}

\newcommand{\n}{{\mathbb N}}
\newcommand{\z}{{\mathbb Z}}
\newcommand{\q}{{\mathbb Q}}
\newcommand{\Mod}[1]{\ (\mathrm{mod}\ #1)}

\begin{document}
%%%%%%%%%%%%%%%%%%%%%%%%%%%%%%%%%%%%%%%%%%%%%%%%%%%%%%%%%%%%%%%%%%%%%%%%%%%%%
%%%%%%%%%%%%%%%%%%%%%%%%%%%%%%%%%%%%%%%%%%%%%%%%%%%%%%%%%%%%%%%%%%%%%%%%%%%%%
\title[Diagonal odd-regular ternary quadratic forms]{Diagonal odd-regular ternary quadratic forms}

\author{Mingyu Kim}

\address{Department of Mathematics, Sungkyunkwan University, Suwon 16419, Korea}
\email{kmg2562@skku.ac.kr}

\thanks{This research was supported by Basic Science Research Program through the National Research Foundation of Korea(NRF) funded by the Ministry of Education(grant number) (NRF-2019R1A6A3A01096245).}


%%% ------------------------------------------------------------------------------------------------------

\begin{abstract} 
A (positive definite primitive integral) quadratic form is called odd-regular if it represents every odd positive integer which is locally represented.
In this paper, we show that there are at most 147 diagonal odd-regular ternary quadratic forms and prove the odd-regularities of all but 6 candidates.
\end{abstract}

\maketitle

%%%%%%%%%%%%%%%%%%%%%%%%%%%%%%%%%%%%%%%%%%%%%%%%%%%%%%%%%%%%%%%%%%%%%%%%%%%%%
%%%%%%%%%%%%%%%%%%%%%%%%%%%%%%%%%%%%%%%%%%%%%%%%%%%%%%%%%%%%%%%%%%%%%%%%%%%%%
\section{Introduction}
%%%%%%%%%%%%%%%%%%%%%%%%%%%%%%%%%%%%%%%%%%%%%%%%%%%%%%%%%%%%%%%%%%%%%%%%%%%%%
%%%%%%%%%%%%%%%%%%%%%%%%%%%%%%%%%%%%%%%%%%%%%%%%%%%%%%%%%%%%%%%%%%%%%%%%%%%%%
In this paper, a quadratic form
$$
f(x_1,x_2,\dots,x_n)=\sum_{1\le i\le j\le n}f_{ij}x_ix_j\quad (f_{ij}\in \z )
$$
always refers to a positive definite (non-classic) integral quadratic form that is primitive in the sense that the coefficients of the quadratic form are coprime, that is,
$$
(f_{11},\dots,f_{ij},\dots,f_{nn})=1.
$$
We say that $f$ is \textit{regular} if it represents every integer which is represented by its genus.
Jones and Pall \cite{JP} gave the list of all diagonal regular ternary quadratic forms. 
Watson proved in his thesis \cite{W1} that there are only finitely many equivalence classes of regular ternary quadratic forms.
Jagy, Kaplansky and Schiemann \cite{JKS} succeeded Watson's study on regular quadratic forms and provide the list of 913 candidates of regular ternary quadratic forms. 
All but 22 of them are already proved to be regular at that time. 
Oh \cite{O1} proved the regularities of 8 forms among remaining 22 candidates.
A conditional proof for the remaining 14 candidates under the Generalized Riemann Hypothesis was given by Lemke Oliver \cite{L}.

For a positive integer $d$ and a nonnegative integer $a$, define a set 
$$
S_{d,a}=\{ dn+a: n\in \z_{\ge 0} \}.
$$
In \cite{O2}, the notion of $S_{d,a}$-regularity is introduced.
A positive definite integral quadratic form $f$ is called \textit{$S_{d,a}$-regular} if the following two conditions hold;
\begin{enumerate} [(i)]
\item $f$ represents every integer in the set $S_{d,a}$ which is represented by its genus;
\item the genus of $f$ represents at least one integer in $S_{d,a}$.
\end{enumerate}
%In this definition, the second condition is needed to exclude the cases of a quadratic form whose genus does not represents any integer in $S_{d,a}$.
In the same paper, it was proved that there is a polynomial $R(x,y)\in \q(x,y)$ such that the discriminant $df$ of any $S_{d,a}$-regular ternary quadratic form $f$ is bounded by $R(d,a)$.
This implies the finiteness of the $S_{d,a}$-regular ternary quadratic forms (up to equivalence) for any given $d$ and $a$.
We say that a quadratic form $f$ is \textit{odd-regular} if it is $S_{2,1}$-regular.

Kaplansky \cite{K} showed that there are exactly 3 diagonal ternary quadratic forms that represent all odd positive integers. They are
$$
x^2+y^2+2z^2,\quad x^2+2y^2+3z^2,\quad x^2+2y^2+4z^2
$$
and by definition, all of these forms are odd-regular.
Note that the above three forms are regular and thus contained in the list of 102 diagonal regular ternary quadratic forms given in \cite{JP}.

In this article, we show that there are at most 147 diagonal odd-regular ternary quadratic forms and prove the regularities of all but 6 candidates.
Since all diagonal regular ternary quadratic forms are odd-regular,
we actually show that there are at most 45 diagonal odd-regular ternary quadratic forms which are not regular and prove the odd-regularities of 39 forms. 
To bound the discriminant of a diagonal odd-regular ternary quadratic forms,
we compute some quantities on the representation of primes in an arithmetic progression by a binary quadratic form.

In 1840, Dirichlet conjectured that any binary quadratic form representing an integer in a given arithmetic progression represents infinitely many primes in that arithmetic progression.
The conjecture was proved by Meyer \cite{M}.
Later in 2003, Halter-Koch \cite[Proposition 1]{Ha} showed that if a binary quadratic form represents an integer in a given arithmetic progression,
then the set of primes in the arithmetic progression represented by the binary quadratic form has positive Dirichlet density.

For a $\z$-lattice $L=\z x_1+\z x_2+\cdots+\z x_k$, the \textit{matrix presentation $M_L$ of $L$} is the matrix
$$
M_L=(B(x_i,x_j))_{1\le i,j\le k},
$$
and the corresponding quadratic form is defined as
$$
\sum_{i,j=1}^k B(x_i,x_j)x_ix_j.
$$
The determinant of the matrix $B((x_i,x_j))$ is called the \textit{discriminant of $L$} and will be denoted by $dL$ as usual.
If $L$ admits an orthogonal basis $\{x_1,x_2,\dots,x_k\}$, then we simply write
$$
L\simeq \langle Q(x_1),Q(x_2),\dots,Q(x_k)\rangle.
$$
For an integer $a$ and a $\z$-lattice $L$, we write $a\ra L$ if $a$ is represented by $L$ and $a\nra L$ otherwise.
We denote by $\gen(L)$ the genus of $L$.
In this article, we abuse the notation when we explicitly write down the isometry classes in $\gen(L)$.
For example, we just write
$$
\gen(L)=\left\{ L,M\right\}
$$
when the class number $h(L)$ of $L$ is two and the genus mate of $L$ is $M$.
We also write $a\ra \gen(L)$ when $a$ is represented by the genus of $L$.
The set of all nonnegative integers represented by $L$ will be denoted by $Q(L)$.
The scale ideal and norm ideal of $L$ is denoted by $\mathfrak{s}(L)$ and $\mathfrak{n}(L)$, respectively.
Throughout the paper, we always assume that every $\z$-lattice is positive definite and  primitive in the sense that $\mathfrak{n}(L)=\z$.
So the scale ideal $\mathfrak{s}(L)$ is $\frac12 \z$ or $\z$.
For an odd prime $p$, we denote by $\Delta_p$ a non-square unit in $\z_p^{\times}$. 
Any unexplained notations and terminologies can be found in \cite{Ki} or \cite{OM}.

%%%%%%%%%%%%%%%%%%%%%%%%%%%%%%%%%%%%%%%%%%%%%%%%%%%%%%%%%%%%%%%%%%%%%%%%%%%%%
%%%%%%%%%%%%%%%%%%%%%%%%%%%%%%%%%%%%%%%%%%%%%%%%%%%%%%%%%%%%%%%%%%%%%%%%%%%%%
\section{Stability}
%%%%%%%%%%%%%%%%%%%%%%%%%%%%%%%%%%%%%%%%%%%%%%%%%%%%%%%%%%%%%%%%%%%%%%%%%%%%%
%%%%%%%%%%%%%%%%%%%%%%%%%%%%%%%%%%%%%%%%%%%%%%%%%%%%%%%%%%%%%%%%%%%%%%%%%%%%%
Let $m$ be a positive integer and let $L$ be a $\z$-lattice.
We define the \textit{Watson transformation of $L$ modulo $m$} to be
$$
\Lambda_m(L)=\{ x\in L : Q(x+z)\equiv Q(z) \Mod m \ \text{for any} \ z\in L \}. 
$$
We denote by $\lambda_m(L)$ the primitive $\z$-lattice obtained from $\Lambda_m(L)$ by scaling the quadratic space $\q \otimes_{\z}L$ by a suitable rational number.

%%%%%%%%%%%%%%%%%%%%%%%%%%%%%%%%%%%%%%%%%%%%%%%%%%%%%%%%%%%%%%%%%%%%%%%%%%%%%
\begin{lem} \label{lemdescend}
Let $p$ be an odd prime.
Let $L$ be a diagonal odd-regular ternary $\z$-lattice such that the unimodular component of $L_p$ is anisotropic.
Then $\lambda_p(L)$ is also a diagonal odd-regular ternary $\z$-lattice. 
\end{lem}

%%%%%%%%%%%%%%%%%%%%%%%%%%%%%%%%%%%%%%%%%%%%%%%%%%%%%%%%%%%%%%%%%%%%%%%%%%%%%
\begin{proof}
Let
$$
L=\langle a,p^rb,p^sc\rangle
$$
such that $(p,abc)=1$ and $0\le r\le s$.
Since the other cases may be treated in a similar manner, we only consider the case when
$r=0$ and $s\ge 2$. Then $\lambda_p(L)\simeq \langle a,b,p^{s-2}c\rangle$. Let $n$ be a nonnegative integer such that
$$
2n+1\ra \gen(\langle a,b,p^{s-2}c\rangle).
$$
Then
$$
p^2(2n+1)\ra \gen(\langle p^2a,p^2b,p^sc\rangle)
$$
and thus
$$
p^2(2n+1)\ra \gen(\langle a,b,p^sc\rangle).
$$
Since $L$ is odd-regular, there is a vector $(x,y,z)\in \z^3$ such that
$$
p^2(2n+1)=ax^2+by^2+p^scz^2.
$$
Since the binary $\z$-lattice $\langle a,b\rangle$ is anisotropic over $\z_p$ by assumption, we have $x\equiv y\equiv 0\Mod p$. Thus
$$
2n+1=a\left(\frac{x}{p}\right)^2+b\left(\frac{y}{p}\right)^2+p^{s-2}cz^2.
$$
This completes the proof.
\end{proof}

%%%%%%%%%%%%%%%%%%%%%%%%%%%%%%%%%%%%%%%%%%%%%%%%%%%%%%%%%%%%%%%%%%%%%%%%%%%%%
Let $L$ be a ternary $\z$-lattice and let $p$ be an odd prime.
We say that $L$ is \textit{$p$-stable} if one of the following holds;
\begin{enumerate} [(i)]
\item $\langle 1,-1\rangle \longrightarrow L_p$;
\item $L_p\simeq \langle 1,-\Delta_p\rangle \perp \langle p\epsilon_p\rangle$ for some $\epsilon_p\in \z_p^{\times}$.
\end{enumerate}
We say that $L$ is \textit{stable} if it is $p$-stable for any odd prime $p$.

Let $L=\langle a,b,c\rangle$ be a diagonal odd-regular ternary $\z$-lattice. 
Then Lemma \ref{lemdescend} says that we may obtain a diagonal stable odd-regular ternary $\z$-lattice $L'=\langle a',b',c'\rangle$ from $L$ by taking $\lambda_q$-transformations several times, if necessary, for suitable odd prime divisors $q$ of the discriminant of $L$.

%%%%%%%%%%%%%%%%%%%%%%%%%%%%%%%%%%%%%%%%%%%%%%%%%%%%%%%%%%%%%%%%%%%%%%%%%%%%%
\begin{lem} \label{lemprime}
Let $L$ be a stable odd-regular ternary $\z$-lattice and let $n$ be a positive odd integer represented by $L_2$.
If $n$ is not represented by $L$, then there is an odd prime divisor $p$ of $n$ such that $L$ is anisotropic over $\z_p$.
\end{lem}

%%%%%%%%%%%%%%%%%%%%%%%%%%%%%%%%%%%%%%%%%%%%%%%%%%%%%%%%%%%%%%%%%%%%%%%%%%%%%
\begin{proof}
Since $L$ is stable, $L_q$ represents any $\gamma \in \z_q$ ($\gamma \in \z_q^{\times}$) if $q$ is an odd prime over which $L$ is isotropic (anisotropic, respectively).
Assume to the contrary that there is no prime divisor $p$ of $n$ such that $L$ is anisotropic over $\z_p$.
Then $n$ is represented by $L_q$ for any odd prime $q$.
Thus $n$ is an odd positive integer which is locally represented by an odd-regular $\z$-lattice $L$.
It follows that $n$ must be represented by $L$, which is absurd.
This completes the proof.
\end{proof}

%%%%%%%%%%%%%%%%%%%%%%%%%%%%%%%%%%%%%%%%%%%%%%%%%%%%%%%%%%%%%%%%%%%%%%%%%%%%%
\begin{lem} \label{lembound1}
Let $w$ be an odd integer and let $p$ be an odd prime.
Let $L=\langle a,b,c\rangle$ be a diagonal ternary $\z$-lattice which is $p$-stable and anisotropic over $\z_p$.
Then there is an integer $g$ satisfying the followings;
\begin{enumerate} [(i)]
\item $0\le g<p^2$.
\item $8g+w \nra \langle a,b\rangle$ over $\z_p$.
\item $8g+w \ra \langle a,b,c\rangle$ over $\z_p$.
\item $\ord_p(8g+w)\le 1$.
\end{enumerate}
\end{lem}

%%%%%%%%%%%%%%%%%%%%%%%%%%%%%%%%%%%%%%%%%%%%%%%%%%%%%%%%%%%%%%%%%%%%%%%%%%%%%
\begin{proof}
First, assume that $p$ divides $c$. 
Since $\langle a,b\rangle$ is isometric to $\langle 1,-\Delta_p\rangle$ over $\z_p$, any element $\gamma \in \z_p$ with $\ord_p\gamma \equiv 1\Mod 2$ is not represented by $\langle a,b\rangle$ over $\z_p$.
We take an integer $g_1$ with $0\le g_1<p^2$ such that $8g_1+w \equiv c\Mod{p^2}$. 
Clearly, $g_1$ satisfies all of the above conditions. 
Next, assume that $p$ divides $ab$. 
Without loss of generality, we may assume that $p$ divides $b$. 
We take an integer $g_2$ with $0\le g_2<p$ such that 
$$
\left( \displaystyle\frac{8g_2+w}{p} \right)=-\left( \displaystyle\frac{a}{p}\right),
$$
where $( \frac{\cdot}{p})$ is the Legendre symbol. 
Then one may easily check that $g_2$ satisfies all of the conditions.
This completes the proof.
\end{proof}

%%%%%%%%%%%%%%%%%%%%%%%%%%%%%%%%%%%%%%%%%%%%%%%%%%%%%%%%%%%%%%%%%%%%%%%%%%%%%
\begin{lem} \label{lembound2}
For an integer $s$ greater than $1$, let $p_1<p_2<\cdots <p_s$ be odd primes. 
Let $u$ be an integer with $(u,p_1p_2\cdots p_s)=1$ and let $v$ be an arbitrary integer.
Then there is an integer $n$ with $0\le n<(s+2)2^{s-1}$ such that 
$(un+v,p_1p_2\cdots p_s)=1$.
\end{lem}

%%%%%%%%%%%%%%%%%%%%%%%%%%%%%%%%%%%%%%%%%%%%%%%%%%%%%%%%%%%%%%%%%%%%%%%%%%%%%
\begin{proof}
This is nothing but \cite[Lemma 3.1]{KO} and is a direct consequence of \cite[Lemma 3]{KKO}.
\end{proof}

%%%%%%%%%%%%%%%%%%%%%%%%%%%%%%%%%%%%%%%%%%%%%%%%%%%%%%%%%%%%%%%%%%%%%%%%%%%%%
\begin{lem} \label{lemineq1}
Let $q_k$ denote the $k$-th odd prime so that $\{q_1=3<q_2=5<q_3=7<\cdots \}$ is the set of odd primes.
Let $N\in \n$ and $\delta \in \{0,1,2,3\}$.
Suppose that $w$ is an integer greater than $\text{max}\left( 4,2\delta \right)$ such that $q_{2\delta+1}q_{2\delta+2}\cdots q_{w}> N\cdot \left( (w+1)2^{w+1}\right)^{\delta}$.
If $p_1<p_2<\cdots<p_s$ are odd primes satisfying $p_1p_2\cdots p_s\le N\cdot \left(p_1^2(s+1)2^{s+1}\right)^{\delta}$, then we have $s\le w-1$.
\end{lem}

%%%%%%%%%%%%%%%%%%%%%%%%%%%%%%%%%%%%%%%%%%%%%%%%%%%%%%%%%%%%%%%%%%%%%%%%%%%%%
\begin{proof}
Assume to the contrary that $s\ge w$. Since
$$
q_i\ge q_6=17>\left( \dfrac{7}{6}\cdot 2\right)^3\ge \left( \dfrac{i+1}{i}\cdot 2\right)^{\delta}
$$
for any $i\ge w+1$, one may easily deduce from the choice of the integer $w$ that
$$
q_{2\delta+1}q_{2\delta+2}\cdots q_s>N\cdot \left( (s+1)2^{s+1}\right)^{\delta}.
$$
Since $p_i\ge q_i$ for any $i$, we have
\begin{align*}
p_1p_2\cdots p_s&\ge p_1^{2\delta}\cdot p_{2\delta+1}p_{2\delta+2}\cdots p_s \\
&\ge p_1^{2\delta}\cdot q_{2\delta+1}q_{2\delta+2}\cdots q_s \\
&> p_1^{2\delta}\cdot N\cdot \left( (s+1)2^{s+1}\right)^{\delta}.
\end{align*}
This is a contradiction.
\end{proof}

%%%%%%%%%%%%%%%%%%%%%%%%%%%%%%%%%%%%%%%%%%%%%%%%%%%%%%%%%%%%%%%%%%%%%%%%%%%%%
For some $N$ and $\delta$, we compute the smallest positive integer $w$ satisfying the conditions of Lemma \ref{lemineq1} and the results sorted in order of appearance in this article are given in Table \ref{tablew}.

%%%%%%%%%%%%%%%%%%%%%%%%%%%%%%%%%%%%%%%%%%%%%%%%%%%%%%%%%%%%%%%%%%%%%%%%%%%%%
\begin{table} [ht]
\caption{}
\begin{tabular}{|c|c|c|}
\hline
$N$ & $\delta$ & $w$ \\
\hline
1&3&21\\
\hline
1&2&11\\
\hline
257&1&8\\
\hline
$193\cdot 401$&0&6\\
\hline
16&1&6\\
\hline
$419\cdot 443$&1&10\\
\hline
$139\cdot 163\cdot 443$&0&8\\
\hline
$389\cdot 397$&1&10\\
\hline
$157\cdot 173\cdot 541$&0&8\\
\hline
$431\cdot 439$&1&10\\
\hline
$167\cdot 191\cdot 431$&0&8\\
\hline
\end{tabular}
\label{tablew}
\end{table}

%%%%%%%%%%%%%%%%%%%%%%%%%%%%%%%%%%%%%%%%%%%%%%%%%%%%%%%%%%%%%%%%%%%%%%%%%%%%%
%%%%%%%%%%%%%%%%%%%%%%%%%%%%%%%%%%%%%%%%%%%%%%%%%%%%%%%%%%%%%%%%%%%%%%%%%%%%%
\section{Representations of odd primes in an arithmetic progression}
%%%%%%%%%%%%%%%%%%%%%%%%%%%%%%%%%%%%%%%%%%%%%%%%%%%%%%%%%%%%%%%%%%%%%%%%%%%%%
%%%%%%%%%%%%%%%%%%%%%%%%%%%%%%%%%%%%%%%%%%%%%%%%%%%%%%%%%%%%%%%%%%%%%%%%%%%%%
For relatively prime positive integers $u$ and $v$, we define a set $P(u,v)$ to be the set of all odd primes in the arithmetic progression $\{un+v\}_{n\ge 0}$. 
Notice that $v$ may be greater than $u$.
A binary quadratic form $M$ is called \textit{$P(u,v)$-universal} if it represents every element of $P(u,v)$.
As noted in \cite{KW}, one may see that
\begin{align*}
P(u,v)=\begin{cases} P(2u,v) & \text{if}\ u\equiv v\equiv 1\Mod 2, \\
P(2u,u+v) & \text{if}\ u\equiv 1,\ v\equiv 0\Mod 2. \end{cases}
\end{align*}
Thus we consider the $P(u,v)$-universality only when $u$ is even and $v$ is odd.
We fix the following notations throughout this section.
In the following two lemmas, we let $M$ be a positive definite integral primitive binary quadratic form $ax^2+bxy+cy^2$ and put $D=D(M)=b^2-4ac$ so that $D(M)=-4dM$.
The product of all odd prime divisors of $D$ will be denoted by $D^*$.

The proof of the following lemma is an easy modification of the proofs of Theorems 1,2 in \cite{KW} but we provide the proof for completeness.

%%%%%%%%%%%%%%%%%%%%%%%%%%%%%%%%%%%%%%%%%%%%%%%%%%%%%%%%%%%%%%%%%%%%%%%%%%%%%
\begin{lem} \label{lemp1}
Assume that $M$ is $P(u,v)$-universal for some relatively prime positive integers $u,v$ with $u$ even and $v$ odd.
Then the class number $h(M)$ of $M$ is 1 and $u$ is divisible by $D^*$.
Furthermore, $M$ is $P(u,v')$-universal, where $v'$ is the smallest positive integer satisfying $v'\equiv v\Mod {u}$.
\end{lem}

%%%%%%%%%%%%%%%%%%%%%%%%%%%%%%%%%%%%%%%%%%%%%%%%%%%%%%%%%%%%%%%%%%%%%%%%%%%%%
\begin{proof}
Suppose that $h(M)>1$. 
Then we can pick a binary $\z$-lattice $N\in \gen(M)$ which is not isometric to $M$. 
Since all primes in $P(u,v)$ are represented by the genus of $M$, it follows that $N$ represents infinitely many primes in $P(u,v)$ by Meyer's theorem \cite{M}.
So there is a prime $p\in P(u,v)$ which is represented by both $M$ and $N$. Since $dM=dN$, one may easily show that this implies that $M$ and $N$ are isometric. 
This contradicts to our choice of $N$ and thus we have $h(M)=1$. 
Assume to the contrary that there is an odd prime divisor $p$ of $D$ such that $(p,u)=1$.
Since $p$ divides $dM$, we may take a positive integer $n'$ with $(p,un'+v)=1$ such that $un'+v \nra M$ over $\z_p$. 
Since $(up,un'+v)=1$, there are infinitely many primes in the arithmetic progression $upn+(un'+v)$ by Dirichlet's theorem on arithmetic progression.
These primes are in $P(u,v)$ and not represented by $M$ which is absurd. 
Thus $u$ is divisible by $D^*$. 
From this, one may easily show that $(q,2D)=1$ for any prime $q\in P(u,v)$.
Since we already have seen that $h(M)=1$, one may easily check that $M$ is $P(u,v')$-universal.
This completes the proof.
\end{proof}

%%%%%%%%%%%%%%%%%%%%%%%%%%%%%%%%%%%%%%%%%%%%%%%%%%%%%%%%%%%%%%%%%%%%%%%%%%%%%
From Lemma \ref{lemp1}, we immediately have

%%%%%%%%%%%%%%%%%%%%%%%%%%%%%%%%%%%%%%%%%%%%%%%%%%%%%%%%%%%%%%%%%%%%%%%%%%%%%
\begin{lem} \label{lemp2}
Let $u$ and $v$ be relatively prime positive integers with $u$ even and $v$ odd.
Assume that $M$ is not $P(u,v)$-universal.
Then there are infinitely many primes in $P(u,v)$ which are not represented by $M$.
\end{lem}

%%%%%%%%%%%%%%%%%%%%%%%%%%%%%%%%%%%%%%%%%%%%%%%%%%%%%%%%%%%%%%%%%%%%%%%%%%%%%
Now we introduce a notion which will be used throughout the paper.
For each $\eta \in \{1,3,5,7\}$, we denote by $q_{\eta,i}$ the $i$-th prime in $P(8,\eta)$ so that $P(8,\eta)=\{ q_{\eta ,1}<q_{\eta, 2}<q_{\eta, 3}<\cdots \}$.
For example, $q_{1,1}=17, q_{1,2}=41, q_{1,3}=73, \cdots$.
We also define $U(8,\eta)$ to be the set of (equivalence classes of) all $P(8,\eta)$-universal diagonal binary quadratic forms.
For positive integers $i,j,u,v,w$ and $\eta \in \{1,3,5,7\}$ with restrictions $i\le j$ and $u\le v$, we define
$$
\xi_{\eta}(i,j;w)=\left\vert \left\{ k\in \n : q_{\eta,k}\ra \langle i,j\rangle,\ k\le w \right\} \right\vert
$$
and we put
$$
\psi_{\eta}(u,v;w)=\text{max}\left\{ \xi_{\eta}(i,j;w) : 1\le i\le u,\ i\le j\le v,\ \langle i,j\rangle \notin U(8,\eta) \right\} .
$$

%%%%%%%%%%%%%%%%%%%%%%%%%%%%%%%%%%%%%%%%%%%%%%%%%%%%%%%%%%%%%%%%%%%%%%%%%%%%%
\begin{lem}
Let $h$ be a positive integer and let $\eta \in \{1,3,5,7\}$.
Then there is an integer $s$ such that $\psi_{\eta}(q_{\eta,h+1},q_{\eta,h+2};s)<s-h$.
\end{lem}

%%%%%%%%%%%%%%%%%%%%%%%%%%%%%%%%%%%%%%%%%%%%%%%%%%%%%%%%%%%%%%%%%%%%%%%%%%%%%
\begin{proof}
Let $K$ be the set defined by
$$
K=\{ (i,j) : 1\le i\le q_{\eta,h+1},\ i\le j\le q_{\eta,h+2},\ \langle i,j\rangle \notin U(8,\eta) \}.
$$
For any $(i,j)\in K$,
$$
w-\xi_{\eta}(i,j;w) \ra \infty \ \ \text{as}\ \ w \ra \infty
$$
by Lemma \ref{lemp2} and thus there is an integer $s_{ij}$ such that
$$
s_{ij}-\xi_{\eta}(i,j;s_{ij})>h.
$$
If we take 
$$
s=\text{max} \{ s_{ij} : (i,j)\in K \},
$$
then we have
$$
s-\xi_{\eta}(i,j;s)>h
$$
for any $(i,j)\in K$. 
Thus
$$
s-\psi_{\eta}(q_{\eta,h+1}, q_{\eta,h+2}; s)>h
$$
and this completes the proof.
\end{proof}

%%%%%%%%%%%%%%%%%%%%%%%%%%%%%%%%%%%%%%%%%%%%%%%%%%%%%%%%%%%%%%%%%%%%%%%%%%%%%
Given $\eta,u,v$ and $w$, one may easily compute $\psi_{\eta}(u,v;w)$ with the help of computer.
We provide the results for some quadruples $(\eta,u,v,w)$ in Table \ref{tablepsi}.
Again, the quadruples are sorted in order of appearance in this paper.

%%%%%%%%%%%%%%%%%%%%%%%%%%%%%%%%%%%%%%%%%%%%%%%%%%%%%%%%%%%%%%%%%%%%%%%%%%%%%
\begin{table} [ht]
\caption{$\psi_{\eta}(u,v;w)$ for some $\eta,u,v$ and $w$.}
\begin{tabular}{|c|c|}
\hline
$(\eta,u,v,w)$ & $\psi_{\eta}(u,v;w)$ \\
\hline
$(1,1,193,16)$ & 8 \\
\hline
$(1,1,73,5)$ & 2\\
\hline
$(3,139,163,22)$ & 12 \\
\hline
$(3,107,131,20)$ & 12 \\
\hline
$(3,67,83,15)$ & 9 \\
\hline
$(3,59,67,12)$ & 7 \\
\hline
$(5,157,173,26)$ & 16 \\
\hline
$(5,53,61,13)$ & 8 \\
\hline
$(7,167,191,21)$ & 11 \\
\hline
$(7,127,151,17)$ & 9 \\
\hline
$(7,71,79,12)$ & 7 \\
\hline
\end{tabular}
\label{tablepsi}
\end{table}

%%%%%%%%%%%%%%%%%%%%%%%%%%%%%%%%%%%%%%%%%%%%%%%%%%%%%%%%%%%%%%%%%%%%%%%%%%%%%
The following is well-known but we provide the simple proof for completeness.

%%%%%%%%%%%%%%%%%%%%%%%%%%%%%%%%%%%%%%%%%%%%%%%%%%%%%%%%%%%%%%%%%%%%%%%%%%%%%
\begin{prop} \label{proppu}
For each $\eta \in \{1,3,5,7\}$, the set $U(8,\eta)$ of all $P(8,\eta)$-universal diagonal binary quadratic forms is as in Table \ref{tableuni}.
\end{prop}

%%%%%%%%%%%%%%%%%%%%%%%%%%%%%%%%%%%%%%%%%%%%%%%%%%%%%%%%%%%%%%%%%%%%%%%%%%%%%
\begin{table} [ht]
\caption{}
\begin{tabular}{|c|c|}
\hline
$\eta$ & $U(8,\eta)$\\
\hline
1 & $\{ \langle 1,1\rangle$, $\langle 1,2\rangle$, $\langle 1,4\rangle$, $\langle 1,8\rangle$, $\langle 1,16\rangle \}$ \\
\hline
3 & $\{ \langle 1,2\rangle \}$ \\
\hline
5 & $\{ \langle 1,1\rangle$, $\langle 1,4\rangle \}$ \\
\hline
7 & $\emptyset$ \\
\hline
\end{tabular}
\label{tableuni}
\end{table}

%%%%%%%%%%%%%%%%%%%%%%%%%%%%%%%%%%%%%%%%%%%%%%%%%%%%%%%%%%%%%%%%%%%%%%%%%%%%%
\begin{proof}
Let $\eta \in \{1,3,5,7\}$ and let $K$ be a $P(8,\eta)$-universal diagonal binary quadratic form.
By Lemma \ref{lemp1}, the discriminant $dK$ of $K$ is a power of 2.
This implies that $K$ is isometric to $\langle 1,2^k\rangle$ for some nonnegative integer $k$.
Note that the class number $h\left( \langle 1,2^k\rangle \right)$ is 1 if and only if $0\le k\le 4$.
Now the proof of $P(8,\eta)$-universality of each forms in the table is straightforward.
This completes the proof.
\end{proof}

%%%%%%%%%%%%%%%%%%%%%%%%%%%%%%%%%%%%%%%%%%%%%%%%%%%%%%%%%%%%%%%%%%%%%%%%%%%%%
%%%%%%%%%%%%%%%%%%%%%%%%%%%%%%%%%%%%%%%%%%%%%%%%%%%%%%%%%%%%%%%%%%%%%%%%%%%%%
\section{Diagonal stable odd-regular ternary quadratic forms}
%%%%%%%%%%%%%%%%%%%%%%%%%%%%%%%%%%%%%%%%%%%%%%%%%%%%%%%%%%%%%%%%%%%%%%%%%%%%%
%%%%%%%%%%%%%%%%%%%%%%%%%%%%%%%%%%%%%%%%%%%%%%%%%%%%%%%%%%%%%%%%%%%%%%%%%%%%%
In this section, we prove that there are exactly 8 diagonal stable odd-regular ternary quadratic forms which are not regular.
We fix some notations that will be used throughout the section.
We denote by $q_k$ the $k$-th odd prime so that $P:=\{ q_1=3<q_2=5<q_3=7<\cdots \}$ is the set of odd primes.
Clearly, we have $P=P(8,1)\cup P(8,3)\cup P(8,5)\cup P(8,7)$.
Let $L=\langle a,b,c\rangle$ $(a\le b\le c)$ be a diagonal stable odd-regular ternary quadratic form and let $T$ be the set of all odd primes at which $L$ is anisotropic.
It is well known that $T$ is always finite and thus we write
$$
T=\{ p_1<p_2<\cdots<p_t\}
$$
so that $t=|T|$. 
For $\eta =1,3,5,7$, we let $T_{\eta}=T\cap P(8,\eta)$ and denote the cardinality of $T_{\eta}$ by $t_{\eta}$.
Clearly, we have $t_{\eta}\le t$ for any $\eta$.
Since $L$ is primitive, at least one of $a,b$ and $c$ is odd.
It follows that there is an odd integer $\alpha \in \{ 1,3,5,7\}$ such that
$$
8n+\alpha \ra \langle a,b,c\rangle \ \ \text{over} \ \ \z_2 
$$
for any nonnegative integer $n$. 

%%%%%%%%%%%%%%%%%%%%%%%%%%%%%%%%%%%%%%%%%%%%%%%%%%%%%%%%%%%%%%%%%%%%%%%%%%%%%
\begin{lem} \label{lemboundofc}
Under the same notations given above, we have $c<p_1^2(t+1)2^{t+1}$.
\end{lem}

%%%%%%%%%%%%%%%%%%%%%%%%%%%%%%%%%%%%%%%%%%%%%%%%%%%%%%%%%%%%%%%%%%%%%%%%%%%%%
\begin{proof}
We first use Lemma \ref{lembound1} when $w=\alpha$ and $p=p_1$ to take an integer $g$ satisfying all the conditions in Lemma \ref{lembound1}. 
Next, use Lemma \ref{lembound2} in the case of $u=8p_1^2$, $v=8g+\alpha$ and odd primes $p_2<p_3<\cdots<p_t$ to obtain an integer $h$ with $0\le h<(t+1)2^{t-2}$ such that $(8p_1^2h+8g+\alpha, p_2p_3\cdots p_t)=1$.
If we put $k=p_1^2h+g$, then from our choice of $g$ and $h$, one may easily see that 
$$
8k+\alpha \nra \langle a,b\rangle \quad \text{and}\quad  8k+\alpha \ra \gen(\langle a,b,c\rangle).
$$ 
From this and the odd-regularity of $\langle a,b,c\rangle$ follows that there exists a vector $(x,y,z)\in \z^3$ with $z\neq 0$ such that
$$
ax^2+by^2+cz^2=8k+\alpha.
$$
So we have $c\le 8k+\alpha$. Since $k\le (p_1^2(t+1)2^{t-2}-1)$ and $\alpha \le 7$, we get 
\begin{align*}
c&\le 8k+\alpha \\
&\le 8(p_1^2(t+1)2^{t-2}-1)+7\\
&< 8p_1^2(t+1)2^{t-2}.
\end{align*}
This completes the proof.
\end{proof}

%%%%%%%%%%%%%%%%%%%%%%%%%%%%%%%%%%%%%%%%%%%%%%%%%%%%%%%%%%%%%%%%%%%%%%%%%%%%%
\begin{lem} \label{lemt20}
Under the same notations given above, we have $t\le 20$.
\end{lem}

%%%%%%%%%%%%%%%%%%%%%%%%%%%%%%%%%%%%%%%%%%%%%%%%%%%%%%%%%%%%%%%%%%%%%%%%%%%%%
\begin{proof}
By Lemma \ref{lemboundofc}, we get
$$
p_1p_2\cdots p_t\le abc\le c^3<(p_1^2(t+1)2^{t+1})^3.
$$
By a direct calculation, one may easily check that $w=21$ satisfies the inequality
$$
q_7q_8\cdots q_w>\left( (w+1)2^{w+1}\right)^3.
$$
Now the lemma follows from Lemma \ref{lemineq1}.
\end{proof}

%%%%%%%%%%%%%%%%%%%%%%%%%%%%%%%%%%%%%%%%%%%%%%%%%%%%%%%%%%%%%%%%%%%%%%%%%%%%%
In the following lemma, the function $\psi_{\eta}(u,v;w)$ is the function defined in Section 3.

%%%%%%%%%%%%%%%%%%%%%%%%%%%%%%%%%%%%%%%%%%%%%%%%%%%%%%%%%%%%%%%%%%%%%%%%%%%%%
\begin{lem} \label{lembound3}
Let $h$ be an integer greater than or equal to $t_{\alpha}$. Then we have
\begin{enumerate}[(i)]
\item for any integer $s>h$, at least $s-h$ primes in the set $\{ q_{\alpha,1},q_{\alpha,2},\dots,q_{\alpha,s}\}$ are represented by $\langle a,b,c\rangle$;
\item $a\le q_{\alpha,h+1}$ and $b\le q_{\alpha,h+2}$;
\item if $a=1$, then $b\le q_{\alpha,h+1}$;
\item if $s$ is an integer such that $\psi_{\alpha}(q_{\alpha,h+1},q_{\alpha,h+2};s)<s-h$, then either $\langle a,b\rangle$ is $P(8,\alpha)$-universal or $c\le q_{\alpha,s}.$
\end{enumerate}
\end{lem}

%%%%%%%%%%%%%%%%%%%%%%%%%%%%%%%%%%%%%%%%%%%%%%%%%%%%%%%%%%%%%%%%%%%%%%%%%%%%%
\begin{proof}
(i) Let $s$ be an integer greater than $h$. Note that 
$$
q_{\alpha,i}\ra \langle a,b,c\rangle \ \ \text{over}\ \ \z_2
$$
for any $i=1,2,\dots,s$ by our choice of $\alpha$. If
$$
q_{\alpha,i}\nra \langle a,b,c\rangle,
$$
for some $i$ with $1\le i\le s$, then we have $q_{\alpha,i}\in T_{\alpha}$ by Lemma \ref{lemprime}.
Thus $\langle a,b,c\rangle$ represents at least $s-t_{\alpha}$ primes in the set $\left\{ q_{\alpha,i} : 1\le i\le s\right\}$. Since $h\ge t_{\alpha}$, we have the assertion.

\noindent(ii) From (i) with $s=h+1$, it follows that there is an integer $i$ with $1\le i\le h+1$ such that $q_{\alpha,i} \ra \langle a,b,c\rangle$. 
Thus we have $a\le q_{\alpha,i}\le q_{\alpha,h+1}$ since we are assuming $a\le b\le c$ throughout this section.
Again from (i) with $s=h+2$, we see that there are two distinct integers $j$ and $k$ with $1\le j<k\le h+2$ such that both $q_{\alpha,j}$ and $q_{\alpha,k}$ are represented by $\langle a,b,c\rangle$.
Assume to the contrary that $b>q_{\alpha,h+2}$. Then two distinct primes $q_{\alpha,j}$ and $q_{\alpha,k}$ must be represented by the unary quadratic form $\langle a\rangle$, which is absurd.
Thus we have $b\le q_{\alpha,h+2}$.

\noindent(iii) As in the proof of (ii), we have
$$
q_{\alpha,i}\ra \langle a,b,c\rangle
$$
for some $i$ with $1\le i\le h+1$.
If $b>q_{\alpha,h+1}$, then we have $q_{\alpha,i}\ra \langle 1\rangle$, which is absurd. Thus $b\le q_{\alpha,h+1}$.

\noindent(iv)
Let $s$ be an integer such that
\begin{equation}\label{eqpsi}
\psi_{\alpha}(q_{\alpha,h+1},q_{\alpha,h+2};s)<s-h.
\end{equation}
Assume to the contrary that $\langle a,b\rangle$ is not $P(8,\alpha)$-universal and $c>q_{\alpha,s}$.
If we define a set $V$ by
$$
V=\left\{ q_{\alpha,i} : q_{\alpha,i}\ra \langle a,b,c\rangle , 1\le i\le s \right\},
$$
then we have $\vert V\vert \ge s-h$ by (i).
Since $c>q_{\alpha,s}$, every prime in $V$ must be represented by $\langle a,b\rangle$.
It follows from this and (ii) that
$$
\psi_{\alpha}(q_{\alpha,h+1},q_{\alpha,h+2};s)\ge s-h
$$
which contradicts to Equation \eqref{eqpsi}. This completes the proof.
\end{proof}

Note that the following 4 lemmas follow from repeated application of Lemma \ref{lemineq1} and Lemmas \ref{lemboundofc}-\ref{lembound3}.

%%%%%%%%%%%%%%%%%%%%%%%%%%%%%%%%%%%%%%%%%%%%%%%%%%%%%%%%%%%%%%%%%%%%%%%%%%%%%
\begin{lem} \label{lemalpha1}
Assume that 1 is represented by $L=\langle a,b,c\rangle$ over $\z_2$. 
Then we have 
$$
(a,b,c)\in \{ (1,4,5), (1,2,24), (1,6,8), (1,5,12), (1,4,21) \}
$$
or $L$ is regular.
\end{lem}

%%%%%%%%%%%%%%%%%%%%%%%%%%%%%%%%%%%%%%%%%%%%%%%%%%%%%%%%%%%%%%%%%%%%%%%%%%%%%
\begin{proof}
Since $L$ is stable and $1\ra L_2$, we have $1\ra \gen(L)$.
Now the odd-regularity of $L$ forces that $1\ra L$ and thus we have $a=1$.

First, we assume that $b\not\in \{ 1,2,4,8,16\}$ so that $\langle 1,b\rangle$ is not $P(8,1)$-universal by Proposition \ref{proppu}.
By Lemma \ref{lemboundofc}, we have
$$
p_1p_2\cdots p_t\le abc<(p_1^2(t+1)2^{t+1})^2.
$$
Now Lemma \ref{lemineq1} with $w=11$ implies $t\le 10$. 
Note that $t_1\le t\le 10$.
By Lemma \ref{lembound3}(iii) with $h=10$, we have
$$
b\le q_{1,11}=257.
$$
Again by Lemma \ref{lemboundofc}, we get
$$
p_1p_2\cdots p_t\le abc<257\cdot (p_1^2(t+1)2^{t+1})
$$
and from this and Lemma \ref{lemineq1} with $w=8$ follows that $t\le 7$. 
By Lemma \ref{lembound3}(iii) with $h=7$, we have $b\le q_{1,8}=193$.
Since $t\le 7$, it follows from Lemma \ref{lembound3}(i) that $L$ represents at least 9 primes in the set $\{ q_{1,i} : 1\le i\le 16\}$.
However, one may easily check that
$$
\psi_1(1,193;16)=8,
$$
which means that
$$
\left\vert \left\{ k : q_{1,k}\ra \langle 1,j\rangle, 1\le k\le 16\right\} \right\vert \le 8
$$
for any $j$ in the set 
$$
J:=\{ 1\le j\le 193 : j\neq 1,2,4,8,16\}.
$$
Since $b\in J$ and $L$ must represent at least 9 primes in the set $\{q_{1,i}:1\le 1\le 16\}$, we have $c\le q_{1,16}=401$.
Now since
$$
p_1p_2\dots p_t\le abc\le 193\cdot 401,
$$
we have $t\le 5$ by Lemma \ref{lemineq1}.
It follows from Lemma \ref{lembound3}(iii) that $b\le q_{1,6}=113$.
Since $q_{1,1}\cdot q_{1,2}=17\cdot 41>401$, we have
$$
\left\vert \left\{ i\in \n : q_{1,i}\vert b \right\} \right\vert \le 1\ \ \text{and}\ \ \left\vert \left\{ i\in \n : q_{1,i}\vert c \right\} \right\vert \le 1.
$$
From this follows that $\left\vert \left\{ i\in \n : q_{1,i}\vert abc \right\} \right\vert \le 2$, and this implies $t_1\le 2$.
So we have $b\le q_{1,3}=73$ by Lemma \ref{lembound3}(iii), and $L$ must represent at least 3 primes in the set $\{ q_{1,i} : 1\le i\le 5\}$ by Lemma \ref{lembound3}(i).
However, one may easily check that
$$
\psi_1(1,73;5)=2
$$
and thus we have $c\le q_{1,5}=97$.
Now with the help of computer, one may check that $(b,c)\in \{ (6,8),(5,12)\}$ or $\langle 1,b,c\rangle$ is regular.

Next, we assume that $b\in \{ 1,2,4,8,16\}$.
From Lemma \ref{lemboundofc} follows that
$$
p_1p_2\cdots p_t\le abc<16\cdot (p_1^2(t+1)2^{t+1}).
$$
By Lemma \ref{lemineq1}, we have $t\le 5$. We assert that $c\le 1633$. Assume to the contrary that $c>1633$. We deal with the case of $b\in \{ 1,4,16\}$ first.

\noindent(i) $b\in \{ 1,4,16\}$. \\
Let $E=\{ 201,553,649,817,1457,1633 \}$.
Note that any element $e$ in the set $E$ is congruent to 1 modulo 8 and is not represented by $\langle 1,b\rangle$. Since $c>1633$, any element of $E$ is not represented by $\langle 1,b,c\rangle$.
Since $\langle 1,b,c\rangle$ is stable odd-regular, there must be an odd prime $p\in T$ such that $p$ divides $e$ by Lemma \ref{lemprime}. 
Since the elements of $E$ are mutually coprime, there should be at least six odd primes in $T$, which is absurd. \\
\noindent(ii) $b\in \{ 2,8\}$. \\
Let $F=\{ 305,553,689,1073,1457,1633\}$. Similarly to the case of $b\in \{ 1,4,16\}$, one may easily deduce a contradiction by considering the set $F$ in place of $E$. \\
Thus we proved the assertion that $c\le 1633$. 
Now with the help of computer, one may easily check that $(b,c)\in \{ (4,5), (2,24), (4,21)\}$ or $\langle 1,b,c\rangle$ is regular.
This completes the proof.
\end{proof}

%%%%%%%%%%%%%%%%%%%%%%%%%%%%%%%%%%%%%%%%%%%%%%%%%%%%%%%%%%%%%%%%%%%%%%%%%%%%%
\begin{lem} \label{lemalpha3}
Assume that, over $\z_2$, $L=\langle a,b,c\rangle$ does not represent 1 and does represent 3. 
Then we have $(a,b,c)\in \{ (3,4,7), (5,6,8) \}$ or $L$ is regular.
\end{lem}

%%%%%%%%%%%%%%%%%%%%%%%%%%%%%%%%%%%%%%%%%%%%%%%%%%%%%%%%%%%%%%%%%%%%%%%%%%%%%
\begin{proof}
We have $t\le 20$ by Lemma \ref{lemt20}.
So Lemma \ref{lembound3}(ii) implies that
$$
a\le q_{3,21}=419\ \ \text{and}\ \  b\le q_{3,22}=443.
$$
Now Lemma \ref{lemboundofc} implies
$$
p_1p_2\cdots p_t\le abc<419\cdot 443\cdot (p_1^2(t+1)2^{t+1}).
$$
Lemma \ref{lemineq1} implies that $t\le 9$.
Thus Lemma \ref{lembound3}(ii) implies
$$
a\le q_{3,10}=139,\ b\le q_{3,11}=163
$$
and
Lemma \ref{lembound3}(i) implies that $L$ represents at least 13 primes in the set 
$$
\{q_{3,i} : 1\le i\le 22\}.
$$
One may easily check that
$$
\psi_3(139,163;22)=12.
$$
Note that $(a,b)\not\eq (1,2)$ by the assumption that
$$
1\nra \langle a,b,c\rangle \ \ \text{over} \ \ \z_2.
$$
Thus $\langle a,b\rangle$ is not $P(8,3)$-universal and thus we may use Lemma \ref{lembound3}(iv) with $h=9$ and $s=22$ to get $c\le q_{3,22}=443$.
From this follows that
$$
p_1p_2\cdots p_t\le abc\le 139\cdot 163\cdot 443.
$$
Lemma \ref{lemineq1} implies $t\le 7$. 
Now Lemma \ref{lembound3}(ii) implies
$$
a\le 107,\ b\le 131
$$
and
Lemma \ref{lembound3}(i) implies that $L$ represents at least 13 primes in the set 
$$
\{ q_{3,i} : 1\le i\le 20\}.
$$
Since
$$
\psi_3(107,131;20)=12\ \ \text{and}\ \  (a,b)\not\eq (1,2),
$$
we have $c\le q_{3,20}=379$ by Lemma \ref{lembound3}(iv).
Note that $q_{3,1}q_{3,2}q_{3,3}=3\cdot 11\cdot 19>379$ and $q_{3,2}q_{3,3}>131$. 
From this and the stability of $\langle a,b,c\rangle$, one may easily deduce that $\vert\{ i : q_{3,i}\vert abc\}\vert \le 5$.
Thus $t_3\le 5$ and now Lemma \ref{lembound3}(ii) implies that
$$
a\le 67,\ b\le 83
$$
and Lemma \ref{lembound3}(i) implies that $L$ represents at least 10 primes in the set 
$$
\{ q_{3,i} : 1\le i\le 15\}.
$$ 
One may easily show that
$$
\psi_3(67,83;15)=9.
$$
Since $(a,b)\not\eq (1,2)$, we have $c\le 251$ by Lemma \ref{lembound3}(iv).
Similarly with above, one may easily deduce that
$$
\left\vert \left\{ i : q_{3,i}\vert abc\right\} \right\vert \le 4.
$$ 
Now Lemma \ref{lembound3}(ii) implies
$$
a\le 59,\ b\le 67
$$
and Lemma \ref{lembound3}(i) implies that $L$ represents at least 8 primes in the set 
$$
\{ q_{3,i} : 1\le i\le 12\}.
$$
Since
$$
\psi_3(59,67;12)=7\ \ \text{and}\ \ (a,b)\not\eq (1,2),
$$
we have $c\le 179$ by Lemma \ref{lembound3}(iv).
Now with the help of computer, one may easily check that $(a,b,c)\in \{ (3,4,7), (5,6,8)\}$ or $L$ is regular. 
This completes the proof.
\end{proof}

%%%%%%%%%%%%%%%%%%%%%%%%%%%%%%%%%%%%%%%%%%%%%%%%%%%%%%%%%%%%%%%%%%%%%%%%%%%%%
Since the proofs of the following two lemmas are quite similar to the proof of Lemma \ref{lemalpha3}, we omit what lemma we use in the proofs.

%%%%%%%%%%%%%%%%%%%%%%%%%%%%%%%%%%%%%%%%%%%%%%%%%%%%%%%%%%%%%%%%%%%%%%%%%%%%%
\begin{lem} \label{lemalpha5}
Assume that, over $\z_2$, $L=\langle a,b,c\rangle$ does not represent 1 and 3 and does represent 5. 
We have $(a,b,c)=(2,5,24)$ or $L$ is regular.
\end{lem}

%%%%%%%%%%%%%%%%%%%%%%%%%%%%%%%%%%%%%%%%%%%%%%%%%%%%%%%%%%%%%%%%%%%%%%%%%%%%%
\begin{proof}
Since $t\le 20$, we have
$$
a\le q_{5,21}=389\ \ \text{and}\ \ b\le q_{5,22}=397.
$$
This implies that
$$
p_1p_2\cdots p_t\le abc<389\cdot 397\cdot (p_1^2(t+1)2^{t+1}).
$$
So we have $t\le 9$ and it follows that $a\le 157$ and $b\le 173$.
Since
$$
\psi_5(157,173;26)=16\ \ \text{and}\ \ (a,b)\not\in \{ (1,1),(1,4)\},
$$
we have $c\le 541$. It follows that
$$
p_1p_2\cdots p_t\le abc\le 157\cdot 173\cdot 541
$$
and thus $t\le 7$.
This implies that $a\le 109$ and $b\le 149$.
Note that 
$$
q_{5,1}q_{5,2}q_{5,3}=5\cdot 13\cdot 29>541\ \ \text{and}\ \ q_{5,2}q_{5,3}>149.
$$
From this, one may deduce that 
$$
\left\vert \left\{ i : q_{5,i}\vert abc\right\} \right\vert \le 4
$$
and thus we have $t_5\le 4$.
Now $a\le 53$, $b\le 61$ and $L$ represents at least 9 primes in the set $\{q_{5,i} : 1\le i\le 13\}$.
It follows from
$$
\psi_5(53,61;13)=8\ \ \text{and}\ \ (a,b)\not\in \{ (1,1),(1,4)\},
$$
that $c\le q_{5,13}=197$. 
With the help of computer, one may easily show that $(a,b,c)=(2,5,24)$ or $L$ is regular. 
This completes the proof.
\end{proof}

%%%%%%%%%%%%%%%%%%%%%%%%%%%%%%%%%%%%%%%%%%%%%%%%%%%%%%%%%%%%%%%%%%%%%%%%%%%%%
\begin{lem} \label{lemalpha7}
Assume that 7 is represented by $L=\langle a,b,c\rangle$ over $\z_2$. 
Then at least one element in the set $\{1,3,5\}$ is also represented by $L$ over $\z_2$ or $L$ is regular.
\end{lem}

%%%%%%%%%%%%%%%%%%%%%%%%%%%%%%%%%%%%%%%%%%%%%%%%%%%%%%%%%%%%%%%%%%%%%%%%%%%%%
\begin{proof}
Since $t\le 20$, we have
$$
a\le q_{7,21}=431\ \ \text{and}\ \ b\le q_{7,22}=439.
$$
It follows that
$$
p_1p_2\cdots p_t\le abc<431\cdot 439\cdot (p_1^2(t+1)2^{t+1})
$$
and thus we have $t\le 9$.
Now $a\le 167$, $b\le 191$ and $L$ represents at least 12 primes in the set 
$$
\{ q_{7,i} : 1\le i \le 21\}.
$$ 
Since there is no $P(8,7)$-universal diagonal binary quadratic form and
$$
\psi_7(167,191;21)=11,
$$
we have $c\le 431$. 
From this follows that
$$
p_1p_2\cdots p_t\le abc\le 167\cdot 191\cdot 431
$$
and thus we have $t\le 7$.
Now $a\le 127$, $b\le 151$ and $L$ represents at least 10 primes in the set 
$$
\{ q_{7,i} : 1\le i\le 17\}.
$$
Since
$$
\psi_7(127,151;17)=9,
$$
we have $c\le q_{7,17}=311$.
From
$$
q_{7,1}q_{7,2}=7\cdot 23>151\ \ \text{and}\ \ q_{7,1}q_{7,2}q_{7,3}=7\cdot 23\cdot 31>311
$$
follows that $t_7\le 4$.
Now we have $a\le 71$, $b\le 79$ and $L$ represents at least 8 primes in the set
$$
\{ q_{7,i} : 1\le i\le 12\}.
$$
Since
$$
\psi_7(71,79;12)=7,
$$
we have $c\le 199$. 
Now with the help of computer, one may show that either $(a,b,c)$ is one of the above 8 triples appeared in Lemmas \ref{lemalpha1}-\ref{lemalpha5} or $L$ is regular.
This completes the proof.
\end{proof}

%%%%%%%%%%%%%%%%%%%%%%%%%%%%%%%%%%%%%%%%%%%%%%%%%%%%%%%%%%%%%%%%%%%%%%%%%%%%%
By Lemmas \ref{lemalpha1}-\ref{lemalpha7}, there are only 8 candidates for diagonal stable odd-regular ternary quadratic forms which are not regular. 
We prove that all of them are indeed odd-regular.
Before that, we need some preparation concerning representations of an arithmetic progression by a ternary quadratic form.
The following approach was introduced in \cite{O1} and some modifications have been made in \cite{O3} and \cite{JOS}. 
Let $l$ be a positive integer and $r$ be a nonnegative integer such that $r<l$. Let $N$ and $K$ be ternary $\z$-lattices and we put
$$
R(N,l,r)=\left\{ v\in \left( \z /d\z \right)^3 : v^tM_Nv\equiv r\Mod l \right\}
$$
and
$$
R(K,N,l)=\left\{ T\in M_3(\z) : T^tM_KT=l^2M_N\right\},
$$
where $M_N$ and $M_K$ are the matrix presentation of $N$ and $K$, respectively.
We say that an element $v\in R(N,l,r)$ is \textit{good} with respect to $K,N,l$ and $r$ if there is a matrix $T\in R(K,N,l)$ such that $\frac1d Tv\in \z^3$.
We put
$$
R_K(N,l,r)=\left\{ v\in R(N,l,r) : v\ \ \text{is good}\right\}.
$$
If $R(N,l,r)=R_K(N,l,r)$, then we write
$$
N\prec_{l,r}K.
$$
We here introduce two results in \cite{O1} and \cite{JOS} for the convenience of the reader.

%%%%%%%%%%%%%%%%%%%%%%%%%%%%%%%%%%%%%%%%%%%%%%%%%%%%%%%%%%%%%%%%%%%%%%%%%%%%%
\begin{lem} \label{lembad1}
Under the same notations given above, assume that $N\prec_{l,r}K$. Then we have
$$
S_{l,r}\cap Q(N)\subset Q(K).
$$
\end{lem}

%%%%%%%%%%%%%%%%%%%%%%%%%%%%%%%%%%%%%%%%%%%%%%%%%%%%%%%%%%%%%%%%%%%%%%%%%%%%%
\begin{proof}
See  \cite[Theorem 2.3]{O1}.
\end{proof}

%%%%%%%%%%%%%%%%%%%%%%%%%%%%%%%%%%%%%%%%%%%%%%%%%%%%%%%%%%%%%%%%%%%%%%%%%%%%%
The following lemma is \cite[Theorem 2.3]{JOS} and will be used in Section 5. For a vector $v\in \z^3$, we denote by $v\Mod l$ the image of $v$ under the natural homomorphism $\z^3\to \left(\z /l\z\right)^3$.

%%%%%%%%%%%%%%%%%%%%%%%%%%%%%%%%%%%%%%%%%%%%%%%%%%%%%%%%%%%%%%%%%%%%%%%%%%%%%
\begin{lem} \label{lembad2}
Under the same notations given above, assume that there is a partition $R(N,l,r)-R_K(N,l,r)=(P_1\cup \cdots \cup P_k)\cup (\widetilde{P_1}\cup \cdots \cup \widetilde{P_{k'}})$ satisfying the following properties: for each $P_i$, there is a $T_i\in M_3(\z)$ such that
\begin{enumerate}[(i)]
\item $\frac1lT_i$ is of infinite order;
\item $T_i^tM_NT_i=l^2M_N$;
\item for any vector $v\in \z^3$ such that $v\Mod l \in P_i$, 
$$
\frac1l T_iv\in \z^3\ \ \text{and}\ \  \frac1lT_iv\Mod d \in P_i\cup R_K(N,l,r)
$$
\end{enumerate}
and for each $\widetilde{P_j}$, there is a $\widetilde{T_j}\in M_3(\z)$ such that
\begin{enumerate} [(i)]
\item $\widetilde{T_j}^tM_N\widetilde{T_j}=l^2M_N$;
\item for any vector $v\in \z^3$ such that $v\Mod l \in \widetilde{P_j}$,
$$
\frac1l\widetilde{T_j}v\in \z^3 \ \ \text{and}\ \ \frac1l\widetilde{T_j}v\Mod l \in P_1\cup \cdots \cup P_k\cup R_K(N,l,r).
$$
\end{enumerate}
Then we have
$$
(S_{l,r}\cap Q(N))-\cup_{i=1}^kg(z_i)\mathcal{S}\subset Q(K),
$$
where the vector $z_i$ is any integral primitive eigenvector of $T_i$, $\mathcal{S}$ is the set of squares of integers, and $g(z_i)\mathcal{S}=\{ g(z_i)n^2: n\in \z\}$.
\end{lem}

%%%%%%%%%%%%%%%%%%%%%%%%%%%%%%%%%%%%%%%%%%%%%%%%%%%%%%%%%%%%%%%%%%%%%%%%%%%%%
\begin{thm} \label{thmstable}
There are exactly 8 diagonal stable odd-regular ternary quadratic forms which are not regular.
$$
\begin{array}{llll}
L(1)=\langle 1,4,5\rangle,&L(2)=\langle 1,2,24\rangle,&L(3)=\langle 1,6,8\rangle,&L(4)=\langle 1,5,12\rangle , \\[0.2em]
L(5)=\langle 1,4,21\rangle,&L(6)=\langle 3,4,7\rangle,&L(7)=\langle 2,5,24\rangle,&L(8)=\langle 5,6,8\rangle .
\end{array}
$$
\end{thm}

%%%%%%%%%%%%%%%%%%%%%%%%%%%%%%%%%%%%%%%%%%%%%%%%%%%%%%%%%%%%%%%%%%%%%%%%%%%%%
\begin{proof}
It only remains to prove the odd-regularities of the 8 forms.
Note that the class number $h(L(i))$ of $L(i)$ is two for any $1\le i\le 8$. 
For each $1\le i\le 8$, we denote the other $\z$-lattice in the genus of $L(i)$ by $M(i)$ so that
$$
gen(L(i))=\{ L(i), M(i) \}.
$$
First, assume that $i\in \{ 1,2,5,7,8 \}$. If $2n+1\ra M(i)$, then one may easily show that there is a sublattice $K(i)$ of $M(i)$ such that
$$
2n+1\ra K(i)\ra L(i),
$$
and for example one may take $K(i)$ as follows.
$$
\begin{array}{ll}
M(1)=\langle 1,1,20\rangle, & K(1)=\langle 1,4,20\rangle, \\[0.5em]
M(2)=\begin{pmatrix} 3 & 1 \\ 1 & 3 \end{pmatrix} \perp \langle 6\rangle, & K(2)=\begin{pmatrix} 3 & 2 \\ 2 & 12 \end{pmatrix} \perp\langle 6\rangle, \\[1em]
M(5)=\langle 1,1,84\rangle, & K(5)=\langle 1,4,84\rangle, \\[0.5em]
M(7)=\begin{pmatrix} 7 & 3 \\ 3 & 7 \end{pmatrix} \perp \langle 6\rangle, & K(7)=\begin{pmatrix} 7 & 6 \\ 6 & 28 \end{pmatrix} \perp\langle 6\rangle, \\[1em]
M(8)=\begin{pmatrix} 11 & 1 \\ 1 & 11 \end{pmatrix} \perp \langle 2\rangle, & K(8)=\begin{pmatrix} 11 & 2 \\ 2 & 44 \end{pmatrix} \perp\langle 2\rangle.
\end{array}
$$
Second, assume that $i=3$. Note that
$$
M(3)=\langle 2\rangle \perp \begin{pmatrix} 4&2\\2&7\end{pmatrix}.
$$
and one may easily check that
$$
M(3)\prec_{l,r}L(3)
$$
for each  pair $(l,r)\in \{ (4,1),(4,3)\}$.
Thus by Lemma \ref{lembad1}, we have
$$
S_{2,1}\cap Q(M(3))\subset Q(L(3)),
$$
and this proves the odd-regularity of $L(3)$.
Third, assume that $i=4$. Note that
$$
M(4)=\begin{pmatrix} 2&0&1\\0&4&2\\1&2&9\end{pmatrix}
$$
and one may check that
$$
M(4)\prec_{l,r}L(4)
$$
for each pair $(l,r)\in \{ (4,3),(8,1),(8,5)\}$.
From this follows the odd-regularity of $L(4)$.
Last, in the case of $i=6$, 
$$
M(6)=\begin{pmatrix} 2&1&1\\1&7&0\\1&0&7\end{pmatrix}
$$
and one may use three pairs
$$
(l,r)=(4,1),(8,3),(8,7)
$$
to show the odd-regularity of $L(6)$. This completes the proof.
\end{proof}

%%%%%%%%%%%%%%%%%%%%%%%%%%%%%%%%%%%%%%%%%%%%%%%%%%%%%%%%%%%%%%%%%%%%%%%%%%%%%
\begin{rmk} \label{rmkprime}
Note that there are exactly 45 diagonal stable regular ternary quadratic forms (for this, see \cite{JP}).
It follows from this and Theorem \ref{thmstable} that there are exactly 53 diagonal stable odd-regular ternary quadratic forms.
One may easily check that if a prime $p$ divides the discriminant of a diagonal stable odd-regular ternary quadratic form, then $p\in \{2,3,5,7\}$.
\end{rmk}

%%%%%%%%%%%%%%%%%%%%%%%%%%%%%%%%%%%%%%%%%%%%%%%%%%%%%%%%%%%%%%%%%%%%%%%%%%%%%
%%%%%%%%%%%%%%%%%%%%%%%%%%%%%%%%%%%%%%%%%%%%%%%%%%%%%%%%%%%%%%%%%%%%%%%%%%%%%
\section{Diagonal odd-regular ternary quadratic forms}
%%%%%%%%%%%%%%%%%%%%%%%%%%%%%%%%%%%%%%%%%%%%%%%%%%%%%%%%%%%%%%%%%%%%%%%%%%%%%
%%%%%%%%%%%%%%%%%%%%%%%%%%%%%%%%%%%%%%%%%%%%%%%%%%%%%%%%%%%%%%%%%%%%%%%%%%%%%
In this section, we prove that there are at most 37 diagonal odd-regular ternary quadratic forms which are not stable and not regular.
And we show the odd-regularities of 31 forms among 37 candidates.
We first introduce the notion of missing primes.
Let $\langle a',b',c'\rangle$ be a diagonal odd-regular ternary quadratic form and let $\langle a,b,c\rangle$ be the stable odd-regular ternary quadratic form obtained from $\langle a',b',c'\rangle$ by taking $\lambda_q$ transformations, if necessary, repeatedly.
It might happen that there is an odd prime divisor $l$ of $a'b'c'$ which does not divide $abc$.
We call such an odd prime $l$ a missing prime.
Note that $\lambda_p \circ \lambda_q=\lambda_q \circ \lambda_p$ for any odd primes $p$ and $q$.
Thus if $l$ is a missing prime, then one of the following holds;
\begin{enumerate} [(i)]
\item $\langle a,l^2b,l^2c\rangle$ is odd-regular;
\item $\langle a,b,l^2c\rangle$ is odd-regular and $\left(\dfrac{-ab}{l}\right)=-1$.
\end{enumerate}

%%%%%%%%%%%%%%%%%%%%%%%%%%%%%%%%%%%%%%%%%%%%%%%%%%%%%%%%%%%%%%%%%%%%%%%%%%%%%
\begin{lem} \label{lemmp}
There is no missing prime greater than 7.
\end{lem}

%%%%%%%%%%%%%%%%%%%%%%%%%%%%%%%%%%%%%%%%%%%%%%%%%%%%%%%%%%%%%%%%%%%%%%%%%%%%%
\begin{proof}
Let $l$ be a missing prime.
First, assume that the quadratic form $\langle a,l^2b,l^2c\rangle$ is odd-regular, where $(abc,l)=1$ and $\langle a,b,c\rangle$ is stable odd-regular.
We assert that $l<157$. Assume to the contrary that $l \ge 157$.
We take an integer $\alpha \in \{ 1,3,5,7\}$ such that $\alpha \ra \langle a,l^2b,l^2c\rangle$ over $\z_2$.
Recall that for an odd prime $p$, $\Delta_p$ is a fixed non-square unit in $\z_p$. Note that
$$
\left\{ \gamma \in \z_p : \gamma \nra \langle 1,-\Delta_p\rangle \perp \langle p\epsilon_p\rangle \right\} 
=\left\{ p^{2k-1}\epsilon_p \Delta_p \left( \z_p^{\times}\right)^2 : k\in \n \right\}
$$
for any odd prime $p$ and an element $\epsilon_p\in \z_p^{\times}$. 
From this, one may easily show that
$$
\left\vert \left\{ 0\le n\le l -1 : 8n+\alpha \nra \langle a,l^2b,l^2c\rangle \ \ \text{over}\ \ \z_p \right\} \right\vert \le \frac{p+1}2\left\lceil \frac{l}{p^2} \right\rceil
$$
for any odd prime $p\neq l$, where $\lceil \cdot \rceil$ is the ceiling function.
Thus we have
$$
\vert \{ 0\le n\le l-1 : 8n+\alpha\nra 
\langle a,l^2 b,l^2 c \rangle \text{ over } \z_p \} \vert \le 
\begin{cases} 
2\left\lceil \displaystyle\frac{l}{9} \right\rceil&\text{if}\ \ p=3, \\[10pt]
3\left\lceil \displaystyle\frac{l}{25} \right\rceil&\text{if}\ \ p=5, \\[10pt]
4\left\lceil \displaystyle\frac{l}{49} \right\rceil&\text{if}\ \ p=7, \\[10pt]
\displaystyle\frac{l+1}{2}&\text{if}\ \ p=l.
\end{cases}
$$
From Theorem \ref{thmstable}, one may check that the set of odd primes at which the stable odd-regular quadratic form $\langle a,b,c\rangle$ is anisotropic is
$$
\emptyset,\ \{ 3\},\ \{ 5\},\ \{ 3,5\}\ \ \text{or}\ \ \{ 3,7\}.
$$
By the assumption that $l\ge 157$, we have $\displaystyle\frac{3}{25}l+3>\frac{4}{49}l+4$.
Thus we have
$$
\left\vert \left\{ 0\le n\le l-1 : 8n+\alpha \ra \gen(\langle a,l^2b,l^2c\rangle \right\} \right\vert \ge \left\lceil l-\left( \frac{2}{9}l+2+\frac{3}{25}l+3+\frac{l+1}{2} \right) \right\rceil.
$$
On the other hand, one may easily show that
$$
\left\vert \left\{ 0\le n\le l-1 : 8n+\alpha \ra \langle a,l^2b,l^2c\rangle \right\} \right\vert \le \left[ \sqrt{2l}+\frac12 \right],
$$
where $[\cdot]$ is the greatest integer function.
From the assumption that $l\ge 157$, one may easily check that
$$
\left\lceil l-\left( \frac{2}{9}l+2+\frac{3}{25}l+3+\frac{l+1}{2} \right) \right\rceil > \left[ \sqrt{2l}+\frac12 \right]
$$
and this is absurd. So we have $l<157$. 
With the help of computer, one may check that the quadratic form $\langle a,l^2b,l^2c\rangle$ is not odd-regular for any $11\le l\le 151$ and any stable odd-regular form $\langle a,b,c\rangle$.

Second, assume that the quadratic form $\langle a,b,l^2c\rangle$ is odd-regular, 
where $a\le b$, $(abc,l)=1$ and $\langle a,b,c\rangle$ is stable odd-regular.
From Theorem \ref{thmstable}, one may check that $(a,b)$ is one of the following 54 pairs;
$$
\begin{array}{ll}
(1,b),& b\in \{ 1,2,3,4,5,6,8,10,12,16,21,24,32,40,48,64\}, \\
(2,b),& b\in \{ 2,3,4,5,6,8,10,16,24,32\}, \\
(3,b),& b\in \{ 4,7,8,10\}, \\
(4,b),& b\in \{ 4,5,6,7,8,12,16,21,24\}, \\
(5,b),& b\in \{ 6,8,12,24\}, \\
(6,b),& b\in \{ 8,16\}, \\
(8,b),& b\in \{ 8,16,24,32,40,64\}, \\
(16,b),& b\in \{ 16,24,48\}.
\end{array}
$$
We assert that $l\le 29$. Assume to the contary that $l>29$.
Fix an integer $\alpha \in \{1,3,5,7\}$ such that $\alpha \ra \gen(\langle a,b,l^2c\rangle)$.
For each $\eta \in \{1,3,5,7\}$, we define a set $E_{\eta}$ as
$$
\begin{array}{ll}
E_1=\{1,11\cdot 19,13\cdot 29\}, \ \ \ &
E_3=\{11,11\cdot 17,13\cdot 23\}, \\[0.3ex]
E_5=\{13,13\cdot 17,11\cdot 23\}, \ \ \ &
E_7=\{ 23,11\cdot 13,17\cdot 23\}.
\end{array}
$$
Note that $E_{\eta}$ consists of positive integers congruent to $\eta$ modulo 8 and that any element in it is not divisible by any odd prime $p$ satisfying $p\le 7$ or $p\ge 31$.
From this and Remark \ref{rmkprime}, every element of $E_{\alpha}$ must be represented by $\gen(\langle a,b,l^2c\rangle)$.
By the odd-regularity of $\langle a,b,l^2c\rangle$, every element of $E_{\alpha}$ is represented by the quadratic form $\langle a,b,l^2c\rangle$ itself.
Since $l>29$, we have that every element of $E_{\alpha}$ must be represented by $\langle a,b\rangle$.
With the help of computer, one may easily check that for any $\eta \in \{1,3,5,7\}$ and for any $(a,b)$ of the 54 pairs given above, there is an element $e\in E_{\eta}$ which is not represented by the binary quadratic form $\langle a,b\rangle$.
This is a contradiction and we showed the assertion.
Now with the help of computer, one may easily check that $\langle a,b,l^2c\rangle$ is not odd-regular for any odd prime $11\le l\le 29$ and for any diagonal stable odd-regular ternary quadratic form $\langle a,b,c\rangle$.
This completes the proof.
\end{proof}

%%%%%%%%%%%%%%%%%%%%%%%%%%%%%%%%%%%%%%%%%%%%%%%%%%%%%%%%%%%%%%%%%%%%%%%%%%%%%
\begin{rmk}
It follows from Remark \ref{rmkprime} and Lemma \ref{lemmp} that if a prime $p$ divides the discriminant of a diagonal odd-regular ternary quadratic form, then $p\in \{2,3,5,7\}$.
\end{rmk}

%%%%%%%%%%%%%%%%%%%%%%%%%%%%%%%%%%%%%%%%%%%%%%%%%%%%%%%%%%%%%%%%%%%%%%%%%%%%%
Following the method described right after Remark 4.2 in \cite{KO}, one may get the list of 147 candidates for diagonal odd-regular ternary quadratic forms including 102 regular forms and 8 stable odd-regular forms in Theorem \ref{thmstable}.
So we have

%%%%%%%%%%%%%%%%%%%%%%%%%%%%%%%%%%%%%%%%%%%%%%%%%%%%%%%%%%%%%%%%%%%%%%%%%%%%%
\begin{thm} \label{thmmain}
There are at most 147 diagonal odd-regular ternary quadratic forms.
\end{thm}

%%%%%%%%%%%%%%%%%%%%%%%%%%%%%%%%%%%%%%%%%%%%%%%%%%%%%%%%%%%%%%%%%%%%%%%%%%%%%
The remaining 37 candidates $L(i)$ $(1\le i\le 37)$ are listed in Table \ref{table37}. In the table, the 6 forms whose odd-regularities are not proved in this article are marked with $\dag$.

%%%%%%%%%%%%%%%%%%%%%%%%%%%%%%%%%%%%%%%%%%%%%%%%%%%%%%%%%%%%%%%%%%%%%%%%%%%%%
\begin{table} [ht]
\caption{37 candidates for diagonal odd-regular ternary quadratic forms}
\begin{tabular}{|lll|}
\hline
 $L(1)=\langle 1,1,36\rangle$,&$L(2)=\langle 1,4,9\rangle$,&$L(3)=\langle 1,5,20\rangle$,\\
\hline
$L(4)=\langle 2,3,24\rangle$,&$L(5)=\langle 3,6,8\rangle$,&$L(6)=\langle 1,3,54\rangle ^{\dag}$,\\
\hline
$L(7)=\langle 3,4,15\rangle$,&$L(8)=\langle 1,12,21\rangle$,&$L(9)=\langle 3,7,12\rangle$,\\
\hline
$L(10)=\langle 1,5,60\rangle$,&$L(11)=\langle 1,9,36\rangle$,&$L(12)=\langle 3,4,27\rangle$,\\
\hline
$L(13)=\langle 1,6,72\rangle$,&$L(14)=\langle 1,18,24\rangle$,&$L(15)=\langle 2,3,72\rangle$,\\
\hline
$L(16)=\langle 2,9,24\rangle$,&$L(17)=\langle 6,8,9\rangle$,&$L(18)=\langle 1,5,100\rangle ^{\dag}$,\\
\hline
$L(19)=\langle 1,12,45\rangle$,&$L(20)=\langle 5,9,12\rangle$,&$L(21)=\langle 1,21,28\rangle$,\\
\hline
$L(22)=\langle 3,7,28\rangle$,&$L(23)=\langle 2,15,24\rangle$,&$L(24)=\langle 6,8,15\rangle$,\\
\hline
$L(25)=\langle 3,15,20\rangle$,&$L(26)=\langle 1,9,108\rangle ^{\dag}$,&$L(27)=\langle 1,16,72\rangle ^{\dag}$,\\
\hline
$L(28)=\langle 1,10,120\rangle$,&$L(29)=\langle 1,30,40\rangle$,&$L(30)=\langle 1,21,84\rangle$,\\
\hline
$L(31)=\langle 3,7,84\rangle$,&$L(32)=\langle 1,45,60\rangle$,&$L(33)=\langle 5,9,60\rangle$,\\
\hline
$L(34)=\langle 1,24,144\rangle ^{\dag}$,&$L(35)=\langle 3,10,120\rangle$,&$L(36)=\langle 3,30,40\rangle$,\\ 
\hline
$L(37)=\langle 3,8,216\rangle ^{\dag}$.&&\\
\hline
\end{tabular}
\label{table37}
\end{table}

%$$
%\begin{array}{llll}
%L(1)=\langle 1,1,36\rangle, & L(2)=\langle 1,4,9\rangle, & L(3)=\langle 1,5,20\rangle, & L(4)=\langle 2,3,24\rangle,  \\
%L(5)=\langle 3,6,8\rangle, & L(6)=\langle 1,3,54\rangle, & L(7)=\langle 3,4,15\rangle, & L(8)=\langle 1,12,21\rangle,  \\
%L(9)=\langle 3,7,12\rangle, & L(10)=\langle 1,5,60\rangle, & L(11)=\langle 1,9,36\rangle, & L(12)=\langle 3,4,27\rangle,  \\
%L(13)=\langle 1,6,72\rangle, & L(14)=\langle 1,18,24\rangle, & L(15)=\langle 2,3,72\rangle, & L(16)=\langle 2,9,24\rangle,  \\
%L(17)=\langle 6,8,9\rangle, & L(18)=\langle 1,5,100\rangle, & L(19)=\langle 1,12,45\rangle, & L(20)=\langle 5,9,12\rangle,  \\
%L(21)=\langle 1,21,28\rangle, & L(22)=\langle 3,7,28\rangle, & L(23)=\langle 2,15,24\rangle, & L(24)=\langle 6,8,15\rangle,  \\
%L(25)=\langle 3,15,20\rangle, & L(26)=\langle 1,9,108\rangle, & L(27)=\langle 1,16,72\rangle, & L(28)=\langle 1,10,120\rangle,  \\
%L(29)=\langle 1,30,40\rangle, & L(30)=\langle 1,21,84\rangle, & L(31)=\langle 3,7,84\rangle, & L(32)=\langle 1,45,60\rangle,  \\
%L(33)=\langle 5,9,60\rangle, & L(34)=\langle 1,24,144\rangle, & L(35)=\langle 3,10,120\rangle, & L(36)=\langle 3,30,40\rangle,  \\
%L(37)=\langle 3,8,216\rangle.&&&
%\end{array}
%$$

%%%%%%%%%%%%%%%%%%%%%%%%%%%%%%%%%%%%%%%%%%%%%%%%%%%%%%%%%%%%%%%%%%%%%%%%%%%%%
\begin{prop} \label{prop1136}
The quadratic form $x^2+y^2+36z^2$ represents every integer congruent to 5 modulo 6 which is represented by the quadratic form $8x^2+9y^2+20z^2+8xz$.
\end{prop}

%%%%%%%%%%%%%%%%%%%%%%%%%%%%%%%%%%%%%%%%%%%%%%%%%%%%%%%%%%%%%%%%%%%%%%%%%%%%%
\begin{proof}
We use Lemma \ref{lembad2} with
$$
K=\langle 1,1,36\rangle, \quad N=\begin{pmatrix} 8&0&4\\0&9&0\\4&0&20 \end{pmatrix}
$$
and $l=3, r=2$. One may easily check that
$R(N,l,r)-R_K(N,l,r)=\{ \pm (1,0,0)\}$. By taking $k=1$, $P_1=\{ \pm (1,0,0)\}$ and
$$
\tau_1=\frac13 \begin{pmatrix} 3&1&2\\0&-1&4\\0&-2&-1\end{pmatrix},
$$
one may easily show that all the conditions in Lemma \ref{lembad2} are satisfied.
Note that $\text{ker}(\tau_1-I)=\langle (1,0,0)\rangle$ and thus it follows that every positive integer congruent to 2 modulo 3 which is not of the form $8k^2 (k\in \z)$ is represented by $K$.
The proposition follows immediately from this.
\end{proof}

%%%%%%%%%%%%%%%%%%%%%%%%%%%%%%%%%%%%%%%%%%%%%%%%%%%%%%%%%%%%%%%%%%%%%%%%%%%%%
\begin{prop} \label{prop1}
The diagonal ternary quadratic form $L(1)=\langle 1,1,36\rangle$ is odd-regular.
\end{prop}

%%%%%%%%%%%%%%%%%%%%%%%%%%%%%%%%%%%%%%%%%%%%%%%%%%%%%%%%%%%%%%%%%%%%%%%%%%%%%
\begin{proof}
Let $n$ be a nonnegative integer such that $2n+1\ra \gen(L(1))$.
Note that the class number $h(L(1))$ of $L(1)$ is two and
$$
\gen(L(1))=\left\{ L(1), M(1)=\langle 1,4,9\rangle \right\}.
$$
We may assume that $2n+1\ra \langle 1,4,9\rangle$ and thus there is a vector $(x,y,z)\in \z^3$ such that
$$
x^2+4y^2+9z^2=2n+1.
$$
We may assume that $x\equiv 0\Mod 2$ since otherwise, we have $z\equiv 0\Mod 2$ and from this follows that
$$
2n+1\ra \langle 1,4,36\rangle \ra \langle 1,1,36\rangle.
$$
If we write $x=2x'$, then $4x'^2+4y^2+9z^2=2n+1$. We may assume that $2n+1\equiv 2\Mod 3$ since otherwise, we have $x'y\equiv 0\Mod 3$ and it follows that
$$
2n+1\ra \langle 4,9,36\rangle \ra \langle 1,1,36\rangle.
$$
Thus we have $x'y\not\equiv 0\Mod 3$ and we may assume that $x'\equiv y\Mod 3$. If we let $x_1=x', y_1=z$ and $z_1=\dfrac{x'-y}3$, then we have
\begin{align*}
2n+1&=4x'^2+4y^2+9z^2 \\
&=4x_1^2+4(x_1-3z_1)^2+9y_1^2 \\
&=8x_1^2+36z_1^2+9y_1^2-24x_1z_1 \\
&=8(x_1-2z_1)^2+9y_1^2+20z_1^2+8(x_1-2z_1)z_1.
\end{align*}
Now the proposition follows immediately from Proposition \ref{prop1136}.
\end{proof}

%%%%%%%%%%%%%%%%%%%%%%%%%%%%%%%%%%%%%%%%%%%%%%%%%%%%%%%%%%%%%%%%%%%%%%%%%%%%%
\begin{prop} \label{prop2372}
The quadratic form $2x^2+3y^2+72z^2$ represents every integer congruent to 1 modulo 4 which is represented by the quadratic form $2x^2+3y^2+18z^2$.
\end{prop}

%%%%%%%%%%%%%%%%%%%%%%%%%%%%%%%%%%%%%%%%%%%%%%%%%%%%%%%%%%%%%%%%%%%%%%%%%%%%%
\begin{proof}
We let
$$
K=\langle 2,3,72\rangle,\ N=\langle 2,3,18\rangle
$$
and $l=4, r=1$. One may easily check that
$$
R(N,l,r)-R_K(N,l,r)=\left\{ \pm (0,1,1),\pm (0,1,3),\pm (2,1,1),\pm (2,1,3) \right\} .
$$
By taking $k=k'=1$, $P_1=\{ \pm (0,1,1),\pm (2,1,3)\}$, $\tilde{P_1}=\{ \pm (0,1,3),\pm (2,1,1)\}$ and
$$
\tau_1=\frac14\begin{pmatrix} 1&3&9\\-2&-2&6\\-1&1&-1\end{pmatrix},\quad 
\widetilde{\tau_1}=\frac14\begin{pmatrix} 1&-3&9\\-2&2&6\\-1&-1&-1\end{pmatrix},
$$
one may easily show that all the conditions in Lemma \ref{lembad2} are satisfied.
Note that $\text{ker}(\tau_1+I)=\langle (0,-3,1)\rangle$.
It follows that every positive integer congruent to 1 modulo 4 which is not of the form $45k^2 (k\in \z)$ is represented by $K$.
Since $K$ does represent 45 as $2\cdot 3^2+3\cdot 3^2=45$, we have the proposition.
\end{proof}

%%%%%%%%%%%%%%%%%%%%%%%%%%%%%%%%%%%%%%%%%%%%%%%%%%%%%%%%%%%%%%%%%%%%%%%%%%%%%
\begin{prop} \label{prop2}
The diagonal ternary quadratic form $L(15)=\langle 2,3,72\rangle$ is odd-regular.
\end{prop}

%%%%%%%%%%%%%%%%%%%%%%%%%%%%%%%%%%%%%%%%%%%%%%%%%%%%%%%%%%%%%%%%%%%%%%%%%%%%%
\begin{proof}
Let $n$ be a nonnegative integer such that $2n+1\ra \gen(L(15))$.
Note that the class number $h(L(15))$ of $L(15)$ is two and
$$
\gen(L(15))=\left\{ L(15), M(15)=\langle 3,8,18\rangle \right\}.
$$
We may assume that $2n+1\ra \langle 3,8,18\rangle$ and thus there is a vector $(x,y,z)\in \z^3$ such that 
$$
3x^2+8y^2+18z^2=2n+1.
$$
If $n\equiv 1\Mod 2$, then we have $z\equiv 0\Mod 2$ and thus 
$2n+1=2(2y)^2+3x^2+72\left(\dfrac{z}2\right)^2$ and we are done. Now we may assume that $n\equiv 0\Mod 2$ and thus $2n+1$ is congruent to 1 modulo 4. Since
$$
2n+1\ra \langle 3,8,18\rangle \ra \langle 2,3,18\rangle,
$$
the proposition follows from Proposition \ref{prop2372}.
\end{proof}

%%%%%%%%%%%%%%%%%%%%%%%%%%%%%%%%%%%%%%%%%%%%%%%%%%%%%%%%%%%%%%%%%%%%%%%%%%%%%
\begin{thm} \label{thmnonstable}
In Table \ref{table37}, $L(i)$ is odd-regular if $i\not\in \{ 6,18,26,27,34,37\}$.
\end{thm}

%%%%%%%%%%%%%%%%%%%%%%%%%%%%%%%%%%%%%%%%%%%%%%%%%%%%%%%%%%%%%%%%%%%%%%%%%%%%%
\begin{proof}
The cases when $i\in\{ 1,15\}$ are treated in Propositions \ref{prop1},\ref{prop2}.
Consider the case when $i=2$. 
Note that $h(L(2))=2$ and
$$
\gen(L(2))=\{ L(2), M(2)=\langle 1,1,36\rangle \}.
$$
Assume that $2n+1\ra M(2)$. Then there is a vector $(x,y,z)\in \z^3$ such that 
$$
2n+1=x^2+y^2+36z^2.
$$
Since $xy\equiv 0\Mod 2$, we have $2n+1\ra \langle 1,4,36\rangle$.
Thus $2n+1\ra \langle 1,1,36\rangle$.
The case when $i=11$ may be shown in a similar manner.
Next, let $i=12$.
Note that $h(L(12))=2$ and
$$
\gen(L(12))=\{ L(12), M(12)=\begin{pmatrix} 7&2\\2&16\end{pmatrix} \perp \langle 3\rangle \}.
$$
Assume that $2n+1\ra M(12)$. One may easily check that this implies $2n+1\equiv 0\Mod 3$ or $2n+1\equiv 7\Mod {12}$.
Now one may check that $M(12)\prec_{l,r}L(12)$ for three pairs of
$$
(r,l)=(3,0), (24,7), (24,19).
$$
Thus it follows from Lemma \ref{lembad1} that $2n+1\ra L(12)$.
The case when $i\in \{ 13,14,16,17,19,20\}$ may be shown in a similar manner as follows.
By using Lemma \ref{lembad1} with pairs
$$
(l,r)=\begin{cases} (4,1),(4,3) & \text{if}\ \ i\in \{ 13,14,16,17\}, \\ (8,1),(8,5) & \text{if}\ \ i\in \{ 19,20\},\end{cases}
$$
one may show that every odd integer which is represented by $M(i)$ is represented by $L(i)$.
For the remaining cases, that is
$$
i\in \{ 3,4,5,7,8,9,10,21,22,23,24,25,28,29,30,31,32,33,35,36\},
$$
one may check that $L(i)$ can be obtained from a diagonal ternary quadratic form which is already proved to be odd-regular by taking $\lambda_p$-transformations several times for some odd primes $p\in \{ 3,5,7\}$.
Moreover, one may see that the odd-regularity is preserved during taking the $\lambda_p$-transformation.
This completes the proof.
\end{proof}

%%%%%%%%%%%%%%%%%%%%%%%%%%%%%%%%%%%%%%%%%%%%%%%%%%%%%%%%%%%%%%%%%%%%%%%%%%%%%
\begin{rmk}
Each of the remaining 6 candidates
$$
\begin{array}{lll}
L(6)=\langle 1,3,54\rangle,& L(18)=\langle 1,5,100\rangle,&L(26)=\langle 1,9,108\rangle,\\[0.3ex]
L(27)=\langle 1,16,72\rangle,&L(34)=\langle 1,24,144\rangle,&L(37)=\langle 3,8,216\rangle
\end{array}
$$
is checked to represent all locally represented odd positive integers up to $10^8-1$ by computer.
\end{rmk}

%%%%%%%%%%%%%%%%%%%%%%%%%%%%%%%%%%%%%%%%%%%%%%%%%%%%%%%%%%%%%%%%%%%%%%%%%%%%%
On the other hand, we may call a ternary $\z$-lattice $L$ \textit{even-regular} if it is $S_{2,0}$-regular.

%%%%%%%%%%%%%%%%%%%%%%%%%%%%%%%%%%%%%%%%%%%%%%%%%%%%%%%%%%%%%%%%%%%%%%%%%%%%%
\begin{thm} \label{thmeven}
Let $L$ be a ternary $\z$-lattice such that $\mathfrak{s}(L)=\z$. Then $L$ is even-regular if and only if $\lambda_2(L)$ is regular.
\end{thm}

%%%%%%%%%%%%%%%%%%%%%%%%%%%%%%%%%%%%%%%%%%%%%%%%%%%%%%%%%%%%%%%%%%%%%%%%%%%%%
\begin{proof}
Since $\mathfrak{s}(L)=\z$, we have
$$
\Lambda_2(L)=\{ x\in L : Q(x)\equiv 0\Mod 2 \}
$$
and it is well known that the local structure $\Lambda_2(L)_q$ of $\Lambda_2(L)$ over $\z_q$ for a prime $q$ is
\begin{align*}
\Lambda_2(L)_q=\begin{cases} \{ x\in L_2: Q(x)\in 2\z_2 \} &\ \ \text{if}\ \ q=2, \\
L_q &\ \ \text{otherwise}.\end{cases}
\end{align*}
Since $\mathfrak{n}(\Lambda_2(L))$ is $2\z$ or $4\z$,
we let $k\in \{2,4\}$ so that $\mathfrak{n}(\Lambda_2(L))=k\z$.
With these facts, one may easily show the following two equivalences;
\begin{align*}
n&\ra \lambda_2(L) \\
\Leftrightarrow kn&\ra \Lambda_2(L) \\
\Leftrightarrow kn&\ra L
\end{align*}
and
\begin{align*}
n&\ra \gen(\lambda_2(L)) \\
\Leftrightarrow kn&\ra \gen(\Lambda_2(L)) \\
\Leftrightarrow kn&\ra \gen(L).
\end{align*}
The theorem follows immediately from this.
\end{proof}

%%%%%%%%%%%%%%%%%%%%%%%%%%%%%%%%%%%%%%%%%%%%%%%%%%%%%%%%%%%%%%%%%%%%%%%%%%%%%
\begin{rmk}
Note that the condition $\mathfrak{s}(L)=\z$ in Theorem \ref{thmeven} is necessary. For example, consider $L=\langle 2\rangle \perp \begin{pmatrix} 1&\frac12 \\ \frac12 & 2\end{pmatrix}$. One may easily check that $L$ is even-regular but $\lambda_2(L)\simeq \langle 1\rangle \perp \begin{pmatrix} 2&1\\1&4\end{pmatrix}$ is not regular.
\end{rmk}

%%%%%%%%%%%%%%%%%%%%%%%%%%%%%%%%%%%%%%%%%%%%%%%%%%%%%%%%%%%%%%%%%%%%%%%%%%%%%

\end{document}